\documentclass[12pt,reqno]{amsart}

\usepackage[colorlinks,linkcolor=red, citecolor=blue, pagebackref=true, pdftex]{hyperref}

\usepackage{fancyvrb}

\usepackage{amssymb, amsmath, amsthm, mathtools, floatflt}
\usepackage[alphabetic,lite]{amsrefs}
\usepackage{verbatim}
\usepackage{tikz}
\usepackage{ mathrsfs }
\usepackage{cases}
\usepackage{verbatimbox}
\usepackage{tabu}
\usepackage{multicol}

\usepackage[margin=1in]{geometry}

\newcommand{\defi}[1]{\textit{\color{blue}#1}}    

\newcommand\inv{inv}

\newcommand\outdeg{\operatorname{outdeg}}   
 
\newcommand\conv{\operatorname{conv}} 
\newcommand\lcm{\operatorname{lcm}} 
\newcommand\Hilb{\operatorname{Hilb}}

\newtheorem{thm}{Theorem}[section]
\newtheorem{prop}[thm]{Proposition}
\newtheorem{lemma}[thm]{Lemma}

\newtheorem{cor}[thm]{Corollary}
\newtheorem*{theorem*}{Theorem}
\newtheorem*{cor*}{Corollary}
\newtheorem*{prop*}{Proposition}

\newtheorem{defn}[thm]{Definition}

\newtheorem{rem}[thm]{Remark}
\newtheorem{example}[thm]{Example}
\newtheorem{question}[thm]{Question}

\begin{document}

\title{One-skeleta of $G$-parking function ideals: \\
resolutions and standard monomials}

\author{Anton Dochtermann}
\address{Department of Mathematics, Texas State University, San Marcos}
\email{dochtermann@txstate.edu}


\date{\today}


\begin{abstract}
Given a graph $G$, the $G$-parking function ideal $M_G$ is an artinian monomial ideal in the polynomial ring $S$ with the property that a linear basis for $S/M_G$ is provided by the set of $G$-parking functions.  It follows that the dimension of $S/M_G$ is given by the number of spanning trees of $G$, which by the Matrix Tree Theorem is equal to the determinant of the reduced Laplacian of $G$. The ideals $M_G$ and related algebras were introduced by Postnikov and Shapiro where they studied their Hilbert functions and homological properties.  The author and Sanyal showed that a minimal resolution of $M_G$ can be constructed from the graphical hyperplane arrangement associated to $G$, providing a combinatorial interpretation of the Betti numbers. 

Motivated by constructions in the theory of chip-firing on graphs, we study certain `skeleton' ideals $M_G^{(k)} \subset M_G$  generated by subsets of vertices of $G$ of size at most $k+1$. Here we focus our attention on the case $k=1$, the $1$-skeleton of the $G$-parking functions ideals. We consider standard monomials of $M_G^{(1)}$ and provide a combinatorial interpretation for the dimension of $S/M_G^{(1)}$ in terms of the signless Laplacian for the case $G = K_{n+1}$ is the complete graph. Our main study concerns homological properties of these ideals. We study resolutions of $M_G^{(1)}$ and show that for a certain class of graphs minimal resolution is supported on decompositions of Euclidean space coming from the theory of tropical hyperplane arrangements.  This leads to combinatorial interpretations of the Betti numbers of these ideals.

\end{abstract}
\maketitle

\section{Introduction}\label{sec:intro}

Let $G = (V,E)$ be a simple undirected graph with vertex set $V = \{0,1, \dots, n\}$, with distinguished sink vertex 0.  In this paper we study various algebraic objects associated with combinatorial properties of $G$.  Our point of departure will be the well-known formula of Cayley that says that the number of spanning trees of the complete graph is equal to $(n+1)^{n-1}$.  This number is related to various other seemingly disparate combinatorial objects, and in particular counts the number of \emph{parking functions} of length $n$.  This correspondence can be generalized to the case of arbitrary graphs $G$ in the context of sequence of integers known as \emph{$G$-parking functions}.

On the algebraic side this data can be encoded in the \emph{$G$-parking function ideal} $M_G$, a monomial ideal living in the polynomial ring $S = {\mathbb K}[x_1, x_2, \dots, x_n]$.  By construction the standard monomials of $M_G$ correspond to $G$-parking functions (this explains the name), and hence are in bijection with the number of spanning trees of $G$.  The standard monomials form a monomial basis for the algebra $S/M_G$ and, in the case of $G = K_{n+1}$ the complete graph, Cayley's formula says that $S/M_G$ has dimension $(n+1)^{n-1}$.   A generating set of monomials for $M_G$ is indexed by all nonempty subsets of $[n] = \{1,2,\dots, n\}$. 
\[M_G = \langle m_\sigma: \emptyset \neq \sigma \subseteq [n] \rangle\]
We refer to the next section for the definition of $m_\sigma$.  For the case of the complete graph $K_{n+1}$ this set gives a minimal set of generators, but for all other graphs $G$ this generating set is redundant.

The ideal $M_G$ was introduced by Postnikov and Shapiro in \cite{PosSha}, where they study various properties of ${\mathcal A}_G = S/M_G$.  In particular they seek to compare ${\mathcal A}_G$ to a related algebra ${\mathcal B}_G$, defined as the quotient of $S$ by certain powers of linear forms also described by the underlying graph $G$.  In \cite{PosSha} it is shown that ${\mathcal A}_G$ and ${\mathcal B}_G$ have the same dimension.  The motivation in \cite{PosSha} came from certain algebras generated by curvature forms on generalized flag manifolds, in an attempt to lift Schubert calculus to the level of differential forms. 

In \cite{PosSha} the authors study homological properties of $M_G$ and show that a cellular resolution of $M_G$ is supported on ${\mathcal B}(\Delta_{n-1})$, the barycentric subdivision of the $(n-1)$-dimensional simplex, which is minimal only for the case of the complete graph $G = K_{n+1}$.  The ideals $M_G$ are examples of more general \emph{monotone} monomial ideals which have cellular resolutions supported on the geometric realization of certain posets. In \cite{DocSan} it is shown that a \emph{minimal} resolution of $M_G$ is supported on a certain polyhedral complex obtained from the graphical hyperplane arrangement associated with $G$.


In this paper we seek to study generalizations of $M_G$ and similarly relate their algebraic properties to well-known combinatorial structures.  Our motivation comes from the theory of \emph{chip-firing} on $G$.  In this context one studies the dynamical system associated with distributing chips on the vertices of a $G$ according to edge adjacency information. The dynamics of chip-firing is governed by the \emph{Laplacian matrix} ${\mathcal L}_G$ of the graph, and it can be shown that the $G$-parking function ideal $M_G$ is in fact a certain initial ideal of the (binomial) lattice ideal determined by ${\mathcal L}_G$. 

Recently several authors have studied the notion of `hereditary set' chip-firing, where certain subsets of vertices are allowed to fire simultaneously (see for instance \cite{Bac} and \cite{CarPaoSpo}).  This rather general notion of chip-firing interpolates between the `abelian sandpile' model (where only singletons fire) and the `cluster model' (where any subset is allowed to fire).  One can see that the `stable' configurations in the cluster model are precisely the $G$-parking functions.  Hereditary set chip-firing shares desirable properties with the more traditional models, including stabilization that is independent of the chosen firings and a well-defined notion of a \emph{recurrent} configuration.   In \cite{Bac} an explicit bijection between these recurrent configurations and the set of spanning trees is described.



Motivated by these constructions,  we study subideals of $M_G$ described by $k$-dimensional `skeleta'.  Recall that $M_G$ has a (possibly redundant) set of generators $m_\sigma$ for every nonempty subset $\sigma \subseteq [n] = \{1,2, \dots, n\}$.  For each integer $k$, define the ideal $M_G^{(k)}$ to be the sub ideal of $M_G$ generated by elements corresponding to subsets of cardinality at most $k+1$:
\[M_G^{(k)} = \langle m_\sigma : \emptyset \neq \sigma \subseteq [n], |\sigma| \leq k+1 \rangle.\]
For $k=0$ the ideal is generated by powers of the variables (corresponding to the degree of the corresponding vertex), and for $k=n-1$ we recover the ideal $M_G$.  In this paper focus on the case $k=1$, the one-dimensional skeleton of $M_G$.  It turns out for this cases there is a story of free resolutions and monomial bases that runs somewhat parallel to that of $M_G$ described above. 

Recall that the standard monomials of $M_G$ are by construction given by the $G$-parking functions (and hence in bijection with the spanning trees of $G$). By the matrix-tree theorem this number is given by the determinant of $\tilde {\mathcal L}_G$, the reduced Laplacian matrix of $G$.  A formula for the number of standard monomials of $M_G^{(k)}$ is not so clear even in the case that $G = K_{n+1}$ is the complete graph.  For the case $k=1$ we can find such a formula, and we see some intriguing connections to other combinatorial objects.  Our main result along these lines are summarized by the following, we refer to subsequent sections for any undefined terms.

\begin{theorem*}[see Theorems \ref{thm:standardoneskel} and Corollary \ref{cor:standardoneskel}]
Let $M_n^{(1)} \coloneqq M_{K_{n+1}}^{(1)}$ denote the $1$-skeleton ideal of the complete graph $K_{n+1}$.   Then the number of standard monomials of $M_{n}^{(1)}$ (and hence $\dim_{\mathbb K} S/M_n^{(1)}$) is given by 
\begin{equation} \label{eq:oneskelstd}
(2n-1)(n-1)^{n-1}.
\end{equation}
This number is given by $\det(\tilde{\mathcal Q}_{K_{n+1}})$, the determinant of the reduced signless Laplacian of the complete graph $K_{n+1}$. 
\end{theorem*}

The value $\det(\tilde{\mathcal Q}_G)$ appearing in (\ref{eq:oneskelstd}) counts (in a weighted fashion) certain combinatorial substructures of the graph $G$, in the same way that the usual reduced Laplacian $\det(\tilde{\mathcal L}_G)$ counts spanning trees (see below for details).  We are unable to find an explicit bijection between the standard monomials of $M_{n}^{(1)}$ and these substructures.  For arbitrary graphs $G$ we illustrate that $\det({\tilde Q}_G)$ does not count standard monomials of $M_G^{(1)}$ (see Example \ref{ex:K5minusedge}), but some computer experimentation suggests that an inequality may hold (see Question \ref{ques:inequality}).

After this paper was posted on the ArXiv we discovered that the standard monomials of $M_n^{(k)}$ can be seen to coincide with a certain class of \emph{${\bf u}$-vector parking functions} as introduced in \cite{PitSta}.  A combinatorial formula for the number of such sequences was established by Yan in \cite{Yan}. See Section \ref{sec:stdmon} and in particular Remark \ref{rem:gen} for more discussion, and also \cite{DocPar} for details regarding the `codimension-one' case of $k = n-2$.  As far as we know the standard monomials of $M_G^{(k)}$ for arbitrary graphs $G$ have no connection to vector parking functions.

We next turn to homological properties of $M_G^{(1)}$, and in particular a combinatorial/geometric description of a minimal free resolution.   Recall in \cite{PosSha} it was shown that a cellular resolution for $M_G$ is supported on the barycentric subdivision of a simplex (with a \emph{minimal} cellular resolution described in \cite{DocSan}). To build resolutions of the skeleton ideals $M_G^{(k)}$, we will search for polyhedral decompositions of the simplex that agree with ${\mathcal B}(\Delta_{n-1})$ on the appropriate subcomplex induced by subsets of cardinality at most $k+1$.  For the case of $k = 1$ there is a natural candidate coming from the the theory of \emph{tropical convexity}.  Our main result is the following.
\begin{theorem*} [see Theorem \ref{thm:resoneskel}]
The ideal $M_{K_{n+1}}^{(1)}$ has a minimal cocellular resolution supported on the labeled polyhedral complex induced by any generic arrangement of two tropical hyperplanes in ${\mathbb R}^{n-1}$.
\end{theorem*}

As was spelled out in \cite{DevStu}, arrangements of tropical hyperplanes are intimately tied to triangulations of a product of simplices, which in turn are related to mixed subdivisions and have many pleasing combinatorial properties \cite{DelRamSan}.  In particular by taking the `staircase triangulation' we obtain a combinatorial interpretation for the Betti numbers of $M_n^{(1)} = M_{K_{n+1}}^{(1)}$ in terms of certain subgraphs of the complete bipartite graph $K_{n,2}$. In addition, we are able to obtain a closed formula for the Betti numbers of these ideals.
\begin{cor*}[See Corollary \ref{cor:totalBetti}]
The total Betti numbers $\beta_i^n$ of the ideal $M_{K_{n+1}}^{(1)}$ are given by
\[\beta_i^n = \sum_{j = 1}^n \; j {j-1 \choose i-1}.\]
\end{cor*}
The complex supporting a resolution of $M_{K_{n+1}}$ does not support a resolution of $M_G^{(1)}$ for arbitrary $G$.  However, we are able to utilize certain degenerations of the underlying arrangement to obtain a minimal resolution in certain cases.

\begin{prop*}[See Proposition \ref{prop:resoneskeldef}]
Suppose $G$ is of the form $G = H \ast \{v\}$, where $H$ is obtained by removing disjoint complete graphs from $K_n$.  Then $M_G^{(1)}$ has a minimal resolution supported on the subdivision of ${\mathbb R}^{n-1}$ induced by a (non-generic) arrangement of two tropical hyperplanes.
\end{prop*}

The rest of the paper is organized as follows.  In Section \ref{sec:basics} we review the basic definitions of the $G$-parking function ideals, their skeleta and related constructions, and also recall some basic facts from combinatorial commutative algebra.  In Section \ref{sec:stdmon} we study the standard monomials of the skeleta ideals, and provide explicit formulas for the case of $M_G^{(1)}$ for the complete graph.  We relate this count to the signless Laplacian and provide a conjectural inequality for arbitrary graphs.  In Section \ref{sec:res} we study homological properties of the skeleta ideal, and in particular establish our results regarding cocellular resolutions of $M_G^{(1)}$ for certain $G$.  We end with some concluding remarks and open questions.

\noindent
{\bf Acknowledgements.}   The results described here grew out of discussions initiated at a workshop on `Generalizations of chip-firing and the critical group' held July 8-12, 2013 at the American Institute of Mathematics.   We thank the organizers for the opportunity to participate and AIM for providing the working environment.  It was at this workshop that Spencer Backman proposed the study of skeleton (and more general) ideals, based on his work involving hereditary chip-firing.  We thank him for introducing us to the subject and for sharing his computations that lead to useful insights. We also thank Sam Hopkins and Suho Oh for useful conversations, as well as two anonymous referees for helpful comments on an earlier draft of the paper.

\section{Definitions and objects of study}
\label{sec:basics}

\subsection{$G$-parking functions, ideals, and their skeleta}
We begin with some basic facts regarding the combinatorial objects involved in our study. Recall that a \defi{parking function of size $n$} is a sequence $(a_1, a_2, \dots, a_n)$ of nonnegative integers such that its rearrangement $c_1 \leq c_2 \leq \cdots \leq c_n$ satisfies $c_i < i$.  This seemingly innocent construction turns out to have connections and applications to many areas of mathematics.  For instance, in \cite{Kre} it is shown that the parking functions of size $n$ are in bijection with the number of trees on $n+1$ labeled vertices which, by Cayley's formula, is given by $(n+1)^{n-1}$.

One can extend these constructions to arbitrary graphs (with the classical case recovered by the complete graph $K_{n+1}$).  For this suppose that $G$ is a simple undirected graph on vertex set $\{0,1,\dots,n\}$ with distinguished `sink' vertex $0$ (much of the theory can be extended to directed graphs with multiple edges, etc. but here we will focus on the simple undirected case).  If $v$ and $w$ are vertices of $G$ we write $v \sim w$ to indicate that $v$ is adjacent to $w$.  We write $\deg(v)$ to denote the \emph{degree} of the vertex $v$, by definition the number of vertices in $G$ that are adjacent to $v$. Associated to a (simple) graph $G$ is the \defi{reduced Laplacian matrx} ${\mathcal L}_G$, an $n \times n$ symmetric matrix with entries given by 

\[(\tilde {\mathcal L}_G)_{i,j} =
\begin{cases}
      \deg(v_i)  & \text{if }i=j \\
      -1 & \text{if } i \neq j \text{ and } v_i \sim v_j \\
      0 & \text{otherwise.}
      \end{cases}
     \]

The \emph{matrix-tree theorem} says that the number $N_G$ of spanning trees of $G$ is given by $N_G = \det \tilde {\mathcal L}_G$.  Our definition of $\tilde {\mathcal L}$ corresponds to deleting from the usual Laplacian ${\mathcal L}_G$ the row and column corresponding to the vertex $0$, but is well known that $N_G$ is independent of which vertex is chosen.  For a subset $\sigma \subseteq [n] = \{1,2,\dots, n\}$ and a vertex $i \in [n]$, let
\begin{equation}\label{eq:dsigma}
d_\sigma(i) = |\{j:  j \sim i, j \notin \sigma\}|
\end{equation}

\begin{defn}
A sequence $(b_1,b_2, \dots, b_n)$ is said to be a \defi{$G$-parking function} if for any $\emptyset \neq \sigma \subseteq [n]$ there exists $i \in \sigma$ such that $b_i < d_S(i)$.  \end{defn}
\noindent
Note that if $G = K_{n+1}$ a complete graph then this recovers the classical parking functions of size $n$ defined above. In addition we have the following enumerative property.

\begin{thm}\cite{Gab}  
The number of $G$-parking functions is given by $\det \tilde {\mathcal L}_G$, the number of spanning trees of $G$.
\end{thm}

In \cite{PosSha} the authors formulate $G$-parking functions in an algebraic context, As above we assume that $G$ is an (simple, undirected) graph on vertex set $\{0,1,\dots,n\}$ with distinguished sink vertex $0$.  Fix a field ${\mathbb K}$ and let $S = {\mathbb K}[x_1, x_2, \dots, x_n]$ denote the polynomial ring on $n$ variables.  For any $\emptyset \neq \sigma \subseteq [n]$  define the monomial $m_\sigma$ by 
\begin{equation}
m_\sigma =  \prod_{i \in \sigma} x_i^{\outdeg_\sigma(i)},
\end{equation}
where $\outdeg_\sigma(i) = \# \{j \in [n] \backslash \sigma : j \sim i\}$ is the number of vertices outside the set $\sigma$ that are adjacent to $i$.

\begin{defn}
For a graph $G$ on vertex set $\{0,1,\dots, n\}$ the $G$-parking function ideal $M_G \subset S$ is generated by all $m_\sigma$ for nonempty $\sigma$:
\[M_G = \langle m_\sigma: \emptyset \neq \sigma \subseteq [n] \rangle.\]
\end{defn}

The ideals $M_G$ were introduced by Postnikov and Shapiro in \cite{PosSha}, where they studied monomial bases for the quotient algebra $S/M_G$ and also described cellular resolutions of these and related ideals.   The ideals $M_G$ also have connections to `chip-firing' and a discrete Riemann-Roch theory of graphs.  In fact it can be shown that $M_G$ is a certain initial ideal of the so-called \emph{toppling ideal} $I_G$, the binomial lattice ideal parametrizing linear equivalence classes of effective divisors on $G$. 

In the context of chip-firing one studies the following dynamical system: a number of `chips' are placed on the vertices of a graph $G$, and a vertex $v$ is allowed to `fire' if the number of chips is at least $\deg(v)$ (the degree of that vertex), in which case one chip is passed to each of its neighbors and $v$ loses $\deg(v)$ chips.  In the \emph{Abelian sandpile model} vertices are restricted to fire individually, whereas in the `cluster firing model' any collection of non-sink vertices is allowed to fire simultaneously.  In both models the stabilization of a configuration is independent of the sequence of firings, and the number of `recurrent' configurations is the same (given by the number of spanning trees).  In the case of the cluster firing method, the recurrent configurations are the same as the stable configurations.

In more recent work several authors (see for instance \cite{Bac}, \cite{CarPaoSpo}) have studied `hereditary chip-firing' models, where the prescribed sets of vertices form a \emph{simplicial complex} on the vertex set of $G$.  In \cite{Bac} it is shown that these models enjoy similar properties to the classical setup, and for example by adapting the Cori-Le Borgne algorithm from \cite{CorLeB} an explicit bijection between the recurrent configurations of a hereditary chip-firing model on a graph and its spanning trees is given. 

In this paper we consider the algebraic aspects of hereditary chip-firing in the context of the monomial ideals defined above.    Although a general theory seems difficult to describe, we do see some pleasing structure arising in the case of subideals defined by `skeleta'.  The following will be our main objects of study.



\begin{defn}
Suppose $G$ is a graph on vertex set $\{0,1,\dots, n\}$. For an integer $k$ with $1 \leq k \leq n-1$, we define the ideal $M_G^{(k)}$ according to
\[M_G^{(k)} = \langle m_\sigma : \emptyset \neq \sigma \subseteq [n], |\sigma| \leq k+1 \rangle.\]
\end{defn}

One can think of the $M_G^{(k)}$ as certain `$k$-skeleta' of the ideal $M_G$. For example for $G = K_5$ and $k=1$ we have
\[M_G^{(1)} = \langle x_1^4, x_2^4, x_3^4, x_4^4, x_1^3x_2^3, x_1^3x_3^3,  x_1^3x_4^3, x_2^3x_4^3, x_2^3x_4^3, x_3^3x_4^3 \rangle, \]
generated by all monomials $m_\sigma$ with $1 \leq |\sigma| \leq 2$.

\noindent
{\bf Notation.}  To simplify notation we will often use $M_n^{(k)} \coloneqq M_{K_{n+1}}^{(k)}$ to denote the $k$-skeleton ideal of the complete graph $G = K_{n+1}$.

We note that the ideals $M_{G}^{(k)}$ are non-generic: for example the monomials $x_1^3x_2^3$ and $x_1^3x_3^3$ are both generators of the ideal $M_{K_5}^{(1)}$ with the same positive degree in $x_1$, and yet the only other generator dividing their least common multiple $x_1^3x_2^3x_3^3$ is the monomial $x_2^3x_3^3$.  Hence certain techniques for studying monomial ideals (as were utilized in \cite{PosSha}) break down.  In addition, the Scarf complex associated to these ideals will not support a resolution.  For the case of a complete graph, the Scarf complex of $M_{n}^{(k)}$ will consist of the $(k-1)$-skeleton of the barycentric subdivision of the $(n-1)$-dimensional simplex, and for $k \leq n-1$ will carry a nontrivial top homology group. 

These ideals are also not \emph{monotone} in the sense of \cite{PosSha} (despite the apparent similarity in the names/construction).  In that context, for any generators $m_\sigma$ and $m_\tau$ one requires that $\lcm(m_\sigma, m_\tau)$ is divisible by a generator $m_\rho$ for some $\rho \supseteq \sigma \cup \tau$.  Our construction of $M_G^{(k)}$ provides a condition on \emph{subsets} of $\sigma$ for each $m_\sigma$, and in this sense runs somewhat dual to the theory of monotone ideals.

\subsection{Resolutions and monomial bases}

We briefly recall some of the commutative algebra needed for our constructions and results.  Further detail and any undefined terms can be found for example in \cite{MilStu}.

Recall that if $M$ is any finely graded $S$-module, then a (${\mathbb Z}^n$-graded) \defi{free resolution} of $M$ is a sequence of free $S$-modules
\[0 \leftarrow M \leftarrow F_1 \leftarrow F_2 \leftarrow \cdots \leftarrow F_d \leftarrow 0,\]
\noindent
where each $F_ i = \bigoplus_{m \in {\mathbb Z}^n} S(-m)^{\beta_{i,m}}$.  The sum is over all monomials $m$ in $S$, which are often (as in this case) described in terms of their exponent vector in ${\mathbb Z}^n$. The resolution is said to be \defi{minimal} if each of the $\beta_{i,m}$ is minimum over all free resolutions of $M$, in which case the $\beta_{i,m}$ are called the (${\mathbb Z}^n$-, or finely-)graded \defi{Betti numbers} of $M$.

A resolution of a monomial ideal $M$ said to be \defi{cellular} (resp. \defi{cocellular}) if there exists a $CW$-complex ${\mathcal X}$ with monomial labelings $m_F$ on its faces $F \in {\mathcal X}$ such that the chain complex computing cellular homology (resp. cohomology) `supports' the resolution.  More specifically, suppose ${\mathcal X}$ is a CW-complex with a monomial associated to each face $\{m_F:F \in {\mathcal X} \}$ satisfying $m_G | m_F$ for any pair of faces $G \subset F$. We define a complex $C_*({\mathcal X})$ of free $S$-modules according to
\[C_i = \bigoplus_{F \in {\mathcal X}, \; \dim F = i+1} S(-m_F)\]
where the sum is over all $(i+1)$-dimensional faces $F$ of ${\mathcal X}$.  The differentials $\partial_i:C_i \rightarrow C_{i-1}$ of $C_*({\mathcal X})$ are monomial-valued matrices with scalar coefficients determined by the entries in the maps of the chain complex of ${\mathbb K}$-modules computing homology of ${\mathcal X}$, and with monomial entries that make the differential a ${\mathbb Z}^n$-graded map.  A cocellular resolution is defined similarly using the cochain complex $C^*({\mathcal X})$ of the underlying space. Note that in the case of a cellular resolution the monomial ideal is generated by the monomials on the vertices (0-cells), whereas in the case of a cocellular complex is generated by the facets (maximal faces).   


Given a monomial labeled complex ${\mathcal X}$ there is a useful criteria for checking where it supports a resolution. For a monomial $m$, let ${\mathcal X}_{\leq m}$ denote the subcomplex of ${\mathcal X}$ consisting of those faces $F$ whose label $m_F$ divides $m$:
\[{\mathcal X}_{\leq m} = \{F \in {\mathcal X}: \text{$m_F$ divides $m$}\}.\]
From \cite{BayStu} we then have the following lemma.

\begin{lemma} 
The complex $C_*({\mathcal X})$ of $S$-modules is exact if and only if the induced complex ${\mathcal X}_{\leq m}$ is ${\mathbb K}$-acyclic for all monomials $m$.  In this case $C_*({\mathcal X})$ supports a cellular resolution of the ideal 
\[I = \langle m_v: \text{$v \in {\mathcal X}$ a vertex}\rangle,\]
\noindent
generated by the monomials on the vertices of ${\mathcal X}$.  If in addition we have that $m_F \neq m_G$ for any proper containment $F \subsetneq G$ of faces then $C_*({\mathcal X})$ is a \emph{minimal} free resolution of $I$.
\end{lemma}

If the labeling on ${\mathcal X}$ satisfies 
\[m_F = \max\{m_G: \text{for $G \supset F$ a face}\}\]
then ${\mathcal X}$ is said to be \defi{colabeled} and gives rise to a cocellular resolution via the cellular cochain complex of ${\mathcal X}$.   Note that for any monomial $m$, the set ${\mathcal X}_{\leq m}$ of cells $H$ with $m_H$ dividing $m$ is not a subcomplex of ${\mathcal X}$, but rather a union of relatively open stars of cells.  The criteria for checking acyclicity of the complex is similar. 

\begin{lemma}\label{lem:cocell}
The colabeled complex $C^*({\mathcal X})$ is exact if and only if  ${\mathcal X}_{\leq m}$ is ${\mathbb K}$-acyclic for every monomial $m \in {\mathbb Z}^n$ (identifying $m$ with its exponent vector).  In this case the $C^*({\mathcal X})$ provides a resolution of  $S/I$, where
\[I = \langle m_H: \text{$H \in {\mathcal X}$ is a maximal cell} \; \rangle.\]
The resolution is minimal if $m_F \neq m_G$ for any proper containment $F \subsetneq G$ of faces.
\end{lemma}

In \cite{PosSha} Postnikov and Shapiro show that the $G$-parking function ideals $M_G$ have cellular resolutions that are supported on ${\mathcal B}(\Delta_n)$, the barycentric subdivision of an $(n-1)$-dimensional simplex.  They generalize this to the case of the order complex of marked posets supporting what they call \emph{order monomial ideals}.  Note that in all cases the underyling complex is a \emph{simplicial} complex.

Suppose $I \in S$ is an ideal such that $S/I$ is finite dimensional as a ${\mathbb K}$-vector space.  Then we say that $I$ is an \emph{Artinian} ideal and a basis for $S/I$ is called the set of \defi{standard monomials} of $I$ (determined by a chosen term order).  In the case that $I$ is monomial, one can see that the standard monomials consist of those monomials $m$ such that $m$ is not divisible by any generators of $I$.  


\section{Monomial bases} \label{sec:stdmon}

Recall that the standard monomials of $M_G$ (a basis for the vector space $S/M_G$) are given by the $G$-parking functions (and hence counted by spanning trees of $G$).  For the case of $G = K_{n+1}$ the number of standard monomials is thus given by $(n+1)^{n-1}$.  In this section we consider the standard monomials of the 1-skeleton ideals, seeking an analogous combinatorial interpretation.   In \cite{DocPar} a formula for the $(n-2)$-skeleton ideal is determined, where a connection to certain `uprooted trees' is also explored.  Note that in general since  $M_G^{(k)} \subset M_G$  the usual $G$-parking functions will be among the standard monomials for the $k$-skeleton ideals for all $k$.  

Before turning to the $1$-skeleton, note that for any graph $G$ the zero-skeleton $M_G^{(0)}$ is generated by each variable raised to the power of the degree of the corresponding vertex.  Hence a monomial basis for $S/M_G^{(0)}$ is easy to describe as the set of monomials lying within an $n$-orthotope with edge lengths given by degree:
\[\Big \{\prod_{i \in [n]} x_i^{d_i}: 0 \leq d_i < \deg(i) \Big\}\]

For the remainder of this section we will mostly focus on the 1-skeleton ideals for the case $G=K_{n+1}$ is a complete graph.  Recall that we use $M_n^{(1)} \coloneqq M_{K_{n+1}}^{(1)}$ to denote the $1$-skeleton ideal of the complete graph $G = K_{n+1}$.

\begin{example}\label{ex:K4oneskel}
For $n=3$ we have $M_3^{(1)} = \langle x_1^3, x_2^3, x_3^3, x_1^2x_2^2, x_1^2x_3^2, x_2^2x_3^2 \rangle$, with the set of standard monomials given by 
\[\{\textrm{$G$-parking functions of length 3}\} \cup \{x_1x_2x_3, x_1^2x_2x_3, x_1x_2^2x_3, x_1x_2x_3^2\},\]
\noindent
giving a total of $16+4 = 20$ standard monomials.   Here as usual we identify a $G$-parking function with the monomial having that sequence as an exponent vector, so that for example $(1,0,2)$ is identified with the monomial $x_1x_3^2$.
\end{example}

In general we have the following formula.

\begin{thm}\label{thm:standardoneskel}
The number of standard monomials of $M_{n}^{(1)}$ (and hence the dimension of the  $\mathbb K$-vector space $S/M_n^{(1)}$) is given by
\[\dim_{\mathbb K}(S/M_n^{(1)}) = (2n-1)(n-1)^{n-1}.\]
\end{thm}

\begin{proof}
Suppose $(a_1, a_2, \dots, a_n)$ is the exponent vector of a standard monomial of $M_n^{(1)}$.  Then by definition each entry $a_i$ is strictly less than $n$, and at most one entry equals $n-1$.  

If no entry equals $n-1$ then every entry satisfies $0 \leq a_i \leq n-2$ so we have $(n-1)^n$ possibilities. If exactly one entry (say $a_i$) equals $n-1$ then every other entry satisfies $0 \leq a_j \leq n-2$ for $j \neq i$.  Hence we have $n(n-1)^{n-1}$ possibilities.  Adding these up we get
\[(n-1)^n + n(n-1)^{n-1} = (2n-1)(n-1)^{n-1}\]
standard monomials, as desired.
\end{proof}

It turns out this number has a determinantal interpretation analogous to the case of $G$-parking functions.  Once again recall that the dimension of $S/M_{K_G}$ (and the number of $G$-parking functions) is equal to $\det \tilde {\mathcal L}_G$, the determinant of the reduced Laplacian of $G$.  For the case of $G = K_{n+1}$ this number is provided by the well known formula $(n+1)^{n-1}$.  For the one-skeleton a different but related matrix makes an appearance.

\begin{defn}
For a graph $G$ on vertex set $\{0,1, \dots, n\}$ the \defi{signless Laplacian} ${\mathcal Q} = {\mathcal Q}_G$ is the symmetric $(n+1) \times (n+1)$ matrix with rows and columns indexed by the vertices of $G$ and with entries given by 
\[{\mathcal Q}_{i,j}= \begin{cases} \deg(i) &\mbox{if } i=j, \\ 
|\{\text{edges connecting $i$ and $j$}\}| & \mbox{if } i \neq j. \end{cases}\]
Define the \defi{reduced signless Laplacian} $\tilde {\mathcal Q}$ to be the matrix obtained from ${\mathcal Q}$ by deleting the row and column corresponding to the vertex 0.
\end{defn}

Note that ${\mathcal Q}$ has entries given by the absolute values of the entries of the usual Laplacian ${\mathcal L}$ (hence the name). For example, if $G=K_4$ we get the following matrices.

\[{\mathcal Q} = \begin{bmatrix}
3 & 1 & 1 &1 \\
1 & 3 & 1 & 1 \\
1 & 1 & 3 & 1\\
1 & 1 & 1 & 3
\end{bmatrix}
\hspace{.5 in}
\tilde {\mathcal Q }= \begin{bmatrix}
3 & 1 & 1 \\
1 & 3 &  1 \\
1 & 1 & 3 
\end{bmatrix}\]
\noindent
In this case one has $\det \tilde {\mathcal Q} = 20 = (5)(2^2)$.  This should remind the reader of the count from Example \ref{ex:K4oneskel} above, and in general we have the following observation.

\begin{cor}\label{cor:standardoneskel}
The number of standard monomials of $M_{n}^{(1)}$ is given by 
\[\dim_{\mathbb K}(S/M_n^{(1)}) = \det \tilde Q_{K_{n+1}}. \]
\end{cor}

\begin{proof}
According to Theorem \ref{thm:standardoneskel} it is enough to show that $\det \tilde{\mathcal Q}_{K_{n+1}} = (2n-1)(n-1)^{n-1}$. For this we examine the eigenvalues of the matrix $\tilde Q_{K_{n+1}}$.  We have one eigenvalue $2n - 1$ with multiplicity 1 corresponding to the all 1's vector ${\bf 1}$.  Subtracting the matrix $(n-1)I_n$ from $\tilde Q_{K_{n+1}}$ gives us the matrix $J$ consisting of all ones, which has an $(n-1)$-dimensional kernel.  Hence $\tilde Q_{K_{n+1}}$ has one other eigenvalue $n-1$ with multiplicity $n-1$.  The result follows.
\end{proof}

It would be interesting to find a bijective proof of Corollary \ref{cor:standardoneskel}, extending the well-known bijections between spanning trees and parking functions that are scattered throughout the literature (see for instance \cite{ChePyl}).  For this we need a graph-theoretical interpretation of the determinant of $\tilde {\mathcal Q}$, which we have (again) thanks to the Cauchy-Binnet theorem.

\begin{prop}\cite{Bap}
For any graph $G$ the determinant of $\tilde {\mathcal Q}_G = \tilde {\mathcal Q}$ is given by 
\[\det \tilde {\mathcal Q} = \sum_H 4^{c(H)}, \]
where the summation runs over all spanning $TU$-subgraphs $H$ of $G$ with $c(H)$ unicyclic components, and one tree component which contains the vertex 0.
\end{prop}

Here a \defi{$TU$-subgraph} is a subgraph of $G$ whose components are trees or unicylic graphs with odd cycles.  Analogous to the usual Laplacian, ${\mathcal Q}$ can be obtained as ${\mathcal Q} = M^T M$, where $M$ is the $0-1$ vertex-edge \emph{signless} incidence matrix of $G$.  Applying Cauchy-Binet to this factorization gives the result (see \cite{Bap} for details).  


A bijective proof of Corollary \ref{cor:standardoneskel} would therefore associate to each spanning $TU$-subgraph $H \subset K_{n+1}$ a collection of $4^{c(H)}$ standard monomials of $M_n^{(1)}$.  Note that each spanning tree of $G$ would be assigned $4^0 = 1$ standard monomials, so presumably such a bijection would extend the correspondence between usual $G$-parking functions and spanning trees.  In Example \ref{ex:K4oneskel} we have the 16 parking functions coming from the spanning trees of $G$, and 4 new standard monomials coming from the $TU$-subgraph consisting of the edges $\{12,13,23\}$.

For graphs that lack symmetry we note that $\det(\tilde Q)$  depends on the choice of sink vertex.  This is in contrast to the usual Laplacian, where the determinant simply counts the number of spanning trees containing the sink and in particular is independent of which row/column is deleted to obtain $\tilde {\mathcal L}$.   In fact Corollary \ref{cor:standardoneskel} does not hold for general $G$, as the following example illustrates. 

\begin{example}\label{ex:K5minusedge}
Let $H$ be the graph obtained from removing the edge $(34)$ from the graph $K_5$.  The reduced signless Laplacian is 
\[\tilde {\mathcal Q}_H = \begin{bmatrix}
4 & 1 & 1 & 1 \\
1 & 4 & 1 & 1 \\
1 & 1 & 3 & 0\\
1 & 1 & 0 & 3
\end{bmatrix}
\]
\noindent
with $\det(\tilde {\mathcal Q}_H) = 99$.  According to Macaulay2 \cite{M2} we have $\dim_{\mathbb K} M_G^{(1)} = 105$, so that there are $105$ standard monomials of $M_G^{(1)}$ in this case. Note however if $H^\prime$ is the graph obtained from $K_5$ by removing the edge $01$ then we get $\det(\tilde {\mathcal Q}_{H^\prime}) = 135$, while there are 135 standard monomials of $M_H^{(1)}$.
\end{example}

\begin{figure} \label{fig:K5minusedge}
\includegraphics[scale = .4]{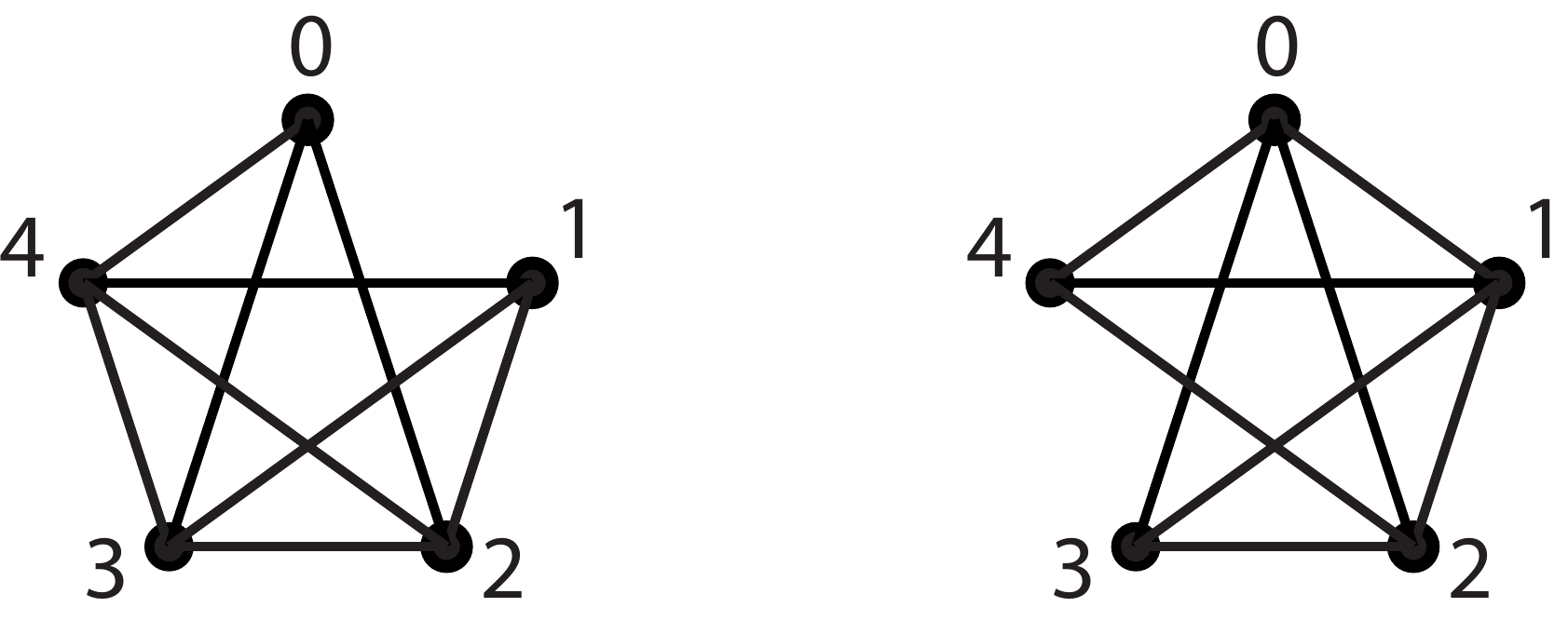}
\caption{The graphs $H^\prime$ and $H$ from Example \ref{ex:K5minusedge}.  The graphs are isomorphic but have different values of $\det \tilde {\mathcal Q}$ (in both cases the sink is given by the vertex 0).  }  
\end{figure}

After a number of calculations we have not found an example where the determinant of $\tilde {\mathcal Q}$ is \emph{larger} than the dimension of $S/M_G^{(1)}$, which begs the following question.

\begin{question} \label{ques:inequality}
For any graph $G$ is it true that
\[\dim_{\mathbb K} S/M_G^{(1)} \geq \det(\tilde {\mathcal Q}_G)?\]
\end{question}
\noindent
The hope here is that for each $TU$ subgraph $H \subset G$ we can find a way to assign $4^{c(H)}$ standard monomials of $M_G^{(1)}$ (in an injective fashion).  

\begin{rem}\label{rem:gen}
After this preprint was posted on the ArXiv we discovered that the standard monomials of $M_n^{(k)}$ can be seen to correspond to a certain class of \emph{vector parking functions}, as introduced by Pitman and Stanley in \cite{PitSta} and further studied by Yan in \cite{YanGen}.  To recall this notion suppose ${\bf u} = (u_1, u_2 , \dots u_n) \in {\mathbb N}^n$ is a vector of nonnegative integers.  A  sequence $(a_1, a_2, \dots, a_n)$ of nonnegative integers is a \defi{${\bf u}$-parking function} if its rearrangement $c_1 \leq c_2 \leq \cdots \leq c_n$ satisfies 
\[c_j < \sum_{i=1}^j u_i\]
for all $1 \leq j \leq n$.

Let $PF({\bf u})$ denote the set of all ${\bf u}$-parking functions. Observe that the usual parking functions are recovered for the case ${\bf u} = (1,1, \dots, 1)$.  One can check that the standard monomials of $M_n^{(k)}$ are naturally in one-to-one correspondence with the set of ${\bf u}_{n,k}$-parking functions, where
\[{\bf u}_{n,k} = (n-k, \underbrace{0, 0 \dots, 0}_\text{$n-k-1$ times}, \underbrace{1,1, \dots, 1}_\text{$k$ times}).\]
\noindent
By results of Yan \cite{Yan} the number of such monomials  is given by
\begin{equation}\label{eq:Yan}
\sum_{j=0}^k {n \choose j} (k + 1 - j)(k+1)^{j-1}(n-k)^{n-j}
\end{equation}
\noindent
For the case $k = 1$ one can check that Equation \ref{eq:Yan} indeed reduces to $(2n-1)(n-1)^{n-1}$, as established in Theorem \ref{thm:standardoneskel}. We refer to \cite{DocPar} for more discussion regarding the case $k = n-2$.
\end{rem}

For an arbitrary graph $G$ (with specified root vertex), the standard monomials of the $k$-skeleton ideal $M_G^{(k)}$ provide a natural blending of the ${\bf u}$-parking functions with the $G$-parking functions from \cite{PosSha}.  One could hope for a generalization Equation \ref{eq:Yan} that counts standard monomials of $M_G^{(k)}$, incorporating data from the underlying graph $G$.  As far as we know these notions have not been explored.


\section{Betti numbers and resolutions}\label{sec:res}

Recall that a free resolution of the $G$-parking function ideal $M_G$ is supported on ${\mathcal B}(\Delta_n)$, the barycentric subdivision of the $(n-1)$-dimensional simplex \cite{PosSha}.  This resolution is minimal only in the case that $G = K_{n+1}$ is the complete graph, in which case the ideal is generic and the resolution can be seen to coincide with the Scarf complex.

For the case of arbitrary $G$ the resolution supported on ${\mathcal B}(\Delta_n)$ is non minimal already in the first homological degree (corresponding to the first Betti number), since if $G$ is missing the edge $ij$ we have that the generator $m_{\{i,j\}}$ is redundant.  Hence the vertices of ${\mathcal B}(\Delta_n)$ do not constitute a set of minimal generators of the ideal $M_G$.   In \cite{DocSan} it was shown that a \emph{minimal} free resolution of $M_G$ is supported on the bounded complex of a certain affine slice of the \emph{graphical hyperplane arrangement} of $G$.  This provides a convenient description of the resolution itself and also leads to a \emph{combinatorial} interpretation of the Betti numbers of $M_G$ in terms of Whitney numbers of the intersection lattice of the graphical arrangement, acyclic orientations of (contractions of) the graph $G$ with unique sink, etc.  We refer to \cite{DocSan} for details.

In the case of the skeleton ideals $M_G^{(k)}$ one can ask whether the Betti numbers have similar combinatorial meaning, and if a (co)cellular resolution can be easily described.  Although in general this question does not seem to have a concise answer, we do have a way to interpret the Betti numbers for a large class of 1-skeleton ideals $M_G^{(1)}$.  

We first point that for any graph $G$ the first Betti number $\beta_1 = \beta_1(M_G^{(1)}$ has a straightforward description. If $G$ has vertex set $V = \{0,1, \dots, n\}$ (with sink $0$) and edge set $E$, we have
\begin{equation}\label{eq:gens}
\beta_1 = n + |E| - \deg(0),
\end{equation}
\noindent
where $\deg(0)$ refers to the degree of the root vertex $0$. To see this note that every vertex $i = 1, 2, \dots, n$ contributes the generator $x_i^{\deg(i)}$.  If $ij \in E$ is an edge in $G$ with $\{i,j\} \subset [n]$ then the subset $\sigma = \{i,j\}$ contributes the generator $m_\sigma = x_i^{\deg(i) - 1}x_j^{\deg(j) - 1}$, which is not generated by the singletons.  On the other hand if $ij \notin E$ we have $m_\sigma = x_i^{\deg(i)}x_j^{\deg(j)}$ which is not needed as generator.  Since by definition $M_G^{(1)}$ is generated by all monomials $m_\sigma$ with $\sigma \subset [n]$ satisfying $1 \leq |\sigma| \leq 2$ the claim follows.

We next turn to the higher syzygies and in particular cellular/cocellular realizations. Any complex supporting a cocellular resolution of $M_G^{(1)}$ must have maximal cells that index nonempty subsets of $[n]$ of cardinality at most 2 (just as the vertices of ${\mathcal B}(\Delta_n)$ are naturally labeled by \emph{all} nonempty subsets of $[n]$).    Indeed such a complex exists, and for its description it will be convenient to use the language of \emph{tropical convexity}, as developed by Develin and Sturmfels in \cite{DevStu} and adapted here for our purposes. Most of our techniques will be standard manipulations in tropical convexity but for the reader unfamiliar with such constructions we will provide a more or less self-contained treatment in what follows.

Let $\Delta_n$ denote the $(n-1)$-dimensional simplex given by the convex hull of the origin and the standard basis vectors
\[\Delta_n = \conv\{{\bf 0}, {\bf e}_1, \dots, {\bf e}_n\}.\]
Now suppose ${\bf a} = (a_1, a_2, \dots, a_n)$ is a point in ${\mathbb R}^n$, thought of as a linear functional in an appropriate dual space.  For us the \defi{tropical hyperplane}  $H({\bf a})$ determined by the functional ${\bf a}$ is by definition the codimension one skeleton of the outward normal fan of $\Delta_n$, with vertex at ${\bf a}$.  Note that the complement of the tropical hyperplane consists of $n+1$ maximal cones which we label $\{1,2, \dots, n+1\}$ according to which vertex of $\Delta_n$ it attains a maximum value on (with the origin treated as the $(n+1)$st vertex). See Figure \ref{fig:planes}.

\begin{figure}[h]
\includegraphics[scale = .95]{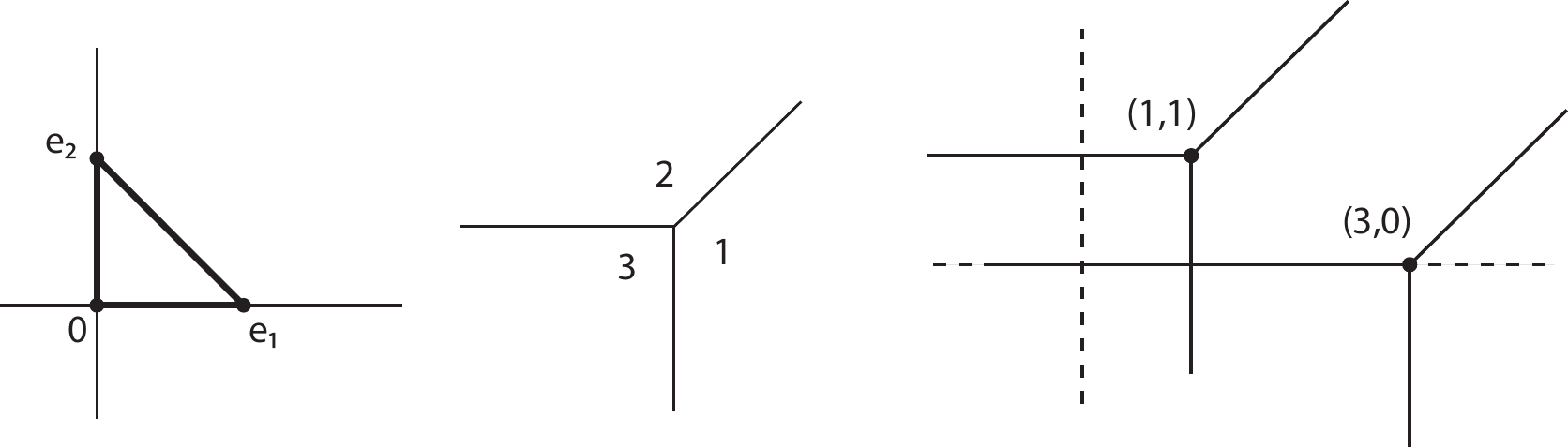}
\caption{The simplex $\Delta_2$, its normal fan with labeled cones, and the arrangement ${\mathcal A} = \{H(1,1), H(3,0)\}$.}
\label{fig:planes}
\end{figure}

A collection ${\mathcal A}$ of tropical hyperplanes then gives rise to a polyhedral decomposition $X_{\mathcal A}$ of ${\mathbb R}^n$, whose relatively open cells are determined by which cones of the hyperplanes they lie in.   We take the following notion from \cite{DevStu}, amended for our purposes.

\begin{defn}
Suppose $(T({\bf a}), T({\bf b}) )$ is an (ordered) arrangement of tropical hyperplanes in ${\mathbb R}^n$, and let ${\bf x} \in {\mathbb R}^n$.  Let ${\bf a}^* = (a_1,  \dots, a_n, a_{n+1}= 0)$, ${\bf b}^* = (b_1, \dots, b_n, b_{n+1}=0)$, and ${\bf x}^* = (x_1, \dots, x_n, x_{n+1}= 0)$ denote the vectors in ${\mathbb R}^{n+1}$ obtained by adding a `0' in the last coordinate.  Then the \defi{type} of ${\bf x}$ is given by $T({\bf x}) = (T_a({\bf x}), T_b({\bf x}))$, where
\[T_a({\bf x}) = \{i \in [n+1]: x_i - a_i = \max_{j \in [n+1]} \{x_j - a_j\}\},\]
\[T_b({\bf x}) = \{i \in [n+1]: x_i - b_i = \max_{j \in [n+1]} \{x_j - b_j\}\}.\]
\end{defn}
\noindent
Alternatively, if we think of each tropical hyperplane in terms of the labeled cones that they define, then $T_a({\bf x})$ (resp. $T_b({\bf x})$)  is the set of closed cones of $H({\bf a})$ (resp. $H({\bf b})$) that intersect the point ${\bf x}$.  We will often use this description in our arguments.

For example in the arrangement depicted in Figure \ref{fig:planes}, we have $T_{\mathcal A}(2,0) = (\{1\}, \{2,3\})$. Note that this notion corresponds to the `fine' types of \cite{DocJosSan}, but we won't need the distinction here.  In \cite{DevStu} it is shown that the set of ${\bf x} \in {\mathbb R}^n$ with a fixed type forms a relatively open convex (and tropically convex) subset of ${\mathbb R}^n$, the cells of the decomposition ${\mathcal X}_{\mathcal A}$.  If $F \in {\mathcal X}_{\mathcal A}$ is a relatively open cell we will often speak of the type of $F$, by which we mean the type of any point ${\bf x} \in F$.

Given a pair of tropical hyperplanes ${\mathcal A} = (H({\bf a}), H({\bf b}))$ we label each maximal cell $F$ of a $X_{\mathcal A}$ with a monomial $m_F$ as follows.   If ${\bf x} \in F$ is any point in the relative interior of $F$ with type $T_{\mathcal A}({\bf x}) = (T_a,T_b)$ we let $\sigma = T_a \cup T_b$ and set $m_F = m_\sigma$.  This is well-defined since, as we have seen, every point in the relative interior of $F$ has the same type.  For example in Figure \ref{fig:planes} the point ${\bf x} = (2,-1)$ has type $(\{1\},\{3\})$ and hence the cell containing ${\bf x}$ is labeled by the monomial
\[m_{\{1,3\}} =   x_1^{\outdeg(1)}x_3^{\outdeg(3)}. \]
If $G \in X_{\mathcal A}$ is any other (non maximal) face, we label it with the monomial $m_G$ where 
\[m_G = \lcm\{m_F: \text{$F$ is a facet with $G \subseteq F$}\}.\]
 \noindent
 We refer to Figure \ref{fig:labelcomplex} for an example.

\begin{figure}[h]
\includegraphics[scale = 1.05]{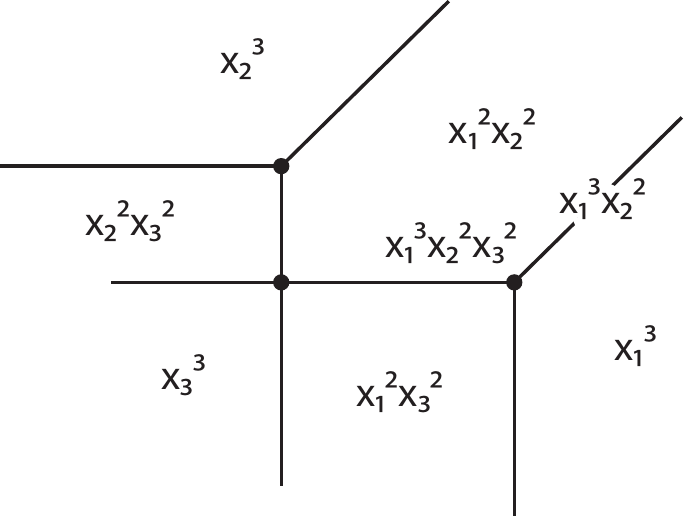} 
\caption{The monomial labeled complex $X_{\mathcal A}$ (with only certain labels depicted).}
 \label{fig:labelcomplex}
\end{figure}

We will see that with this monomial labeling the complex $X_{\mathcal A}$ in fact supports a cocellular resolution of our ideals.  To establish this we will need one more notion from tropical convexity.

\begin{defn}
Given points ${\bf x}, {\bf y} \in {\mathbb R}^n$ the \defi{(max-)tropical line segment} they determine is given by the set of points ${\bf p} \in {\mathbb R}^n$ of the form
\[{\bf p} = \max\{\lambda {\bf 1} + {\bf x}, \mu {\bf 1} + {\bf y}\},\]
\noindent
where $\lambda, \mu$ are arbitary real numbers, and ${\bf 1}$ is the vector in ${\mathbb R}^n$ consisting of all one's.
\end{defn}

We then have the following results from \cite{DocJosSan}, again adapted to suit our setup.

\begin{prop}\cite{DocJosSan} \label{prop:types}
Let ${\mathcal A} = \{H({\bf a}), H({\bf b})\}$ be an arrangement of tropical hyperplanes in ${\mathbb R}^n$, and let ${\mathcal X}_{\mathcal A}$ denote the induced decomposition of ${\mathbb R}^n$. For every cell $G \in X_{\mathcal A}$ of codimension $\geq 1$ we have
\[T_{\mathcal A}(G) = \bigcup_{G \subset F} T_{\mathcal A} (F),\]
\noindent
where the union of a pair of sets is taken to be the pair of the unions, so that e.g. $(S_a, S_b) \cup (T_a, T_b) = (S_a \cup T_a, S_b \cup T_b)$.
\end{prop}

\begin{lemma}\cite{DocJosSan} 
Suppose ${\bf x}, {\bf y} \in {\mathbb R}^n$, and let ${\bf z}$ be a point on the tropical line segment between ${\bf x}$ and ${\bf y}$.  Suppose ${\mathcal A} = \{H({\bf a}), H({\bf b})\}$ is an arrangement of a pair of tropical hyperplanes in ${\mathbb R}^n$, and let $T_{\mathcal A}({\bf x}) = (C_a,C_b)$, $T_{\mathcal A}({\bf y}) = (D_a,D_b)$, and $T_{\mathcal A}({\bf z}) = (E_a,E_b)$ denote the types of these points with respect to ${\mathcal A}$. Then we have the following inclusions:
\begin{equation}\label{typecontain}
\begin{aligned}
&C_a \cap D_a \subseteq E_a \subseteq C_a \cup D_a, \\
&C_b \cap D_b \subseteq E_b \subseteq C_b \cup D_b.
\end{aligned}
\end{equation}
\label{lem:types}
\end{lemma}

\begin{prop}\cite{DocJosSan} \label{prop:distinct}
Let $C$ and $D$ be distinct cells in ${\mathcal X}_{\mathcal A}$ with types $T(C)_{\mathcal A} = (C_a, C_b)$ and $T(D)_{\mathcal A} = (D_a, D_b)$.  If $C$ is contained in the closure of $D$ then as multisets we have
\[C_a \sqcup C_b \neq D_a \sqcup D_b.\]
\end{prop}

Note that cells that are `far away' in the complex ${\mathcal X}_{\mathcal A}$ can violate the inequality in the previous Proposition. For instance in the arrangement in Figure (\ref{fig:planes}) we have $T(1,1/2) = (\{1,3\},\{2\})$ and $T(2,0) = (\{1\}, \{2,3\})$. In this case both of the relevant faces have the monomial label $x_1^2x_2^2x_3^2$.  With these in statements in place we can now prove our main results of this section.

\begin{thm} \label{thm:resoneskel}
Suppose ${\mathcal A} = \{H({\bf a}), H({\bf b}) \}$ is an arrangement of two tropical hyperplanes in ${\mathbb R}^{n-1}$ in generic position, and let $X_{\mathcal A}$ denote the induced subdivision of ${\mathbb R}^{n-1}$ with the monomial labeling described above. Then $X_{\mathcal A}$ supports a minimal cocellular resolution of the ideal $M_{K_{n+1}}^{(1)}$.
\end{thm}

\begin{proof}
Recall that a non-redundant list of generators of the ideal $M_n^{(1)} = M_{K_{n+1}}^{(1)}$ is indexed by all subsets $\sigma$ of $[n] = \{1,2,\dots, n\}$ such that $1 \leq |\sigma| \leq 2$.   Also recall that each maximal cell $F$ of $X_{\mathcal A}$ is labeled by the monomial $m_\sigma$, where $\sigma = T_a \cup T_b$ is the union of the sets occurring as the type of any ${\bf x} \in F$.  Since ${\bf x}$ does not lie on $H({\bf a}) \cup H({\bf b})$, we have $\sigma = \{i,j\}$ with $1 \leq i \leq j \leq n$.  Since the arrangement is assumed to be generic we have that every such (possibly degenerate) pair is realized as the union of type sets for some point ${\bf x} \in {\mathbb R}^n$.    This is equivalent to the fact that in any regular fine mixed subdivision of $\Delta_{n-1} + \Delta_{n-1}$ every lattice point occurs as a 0-cell, and these lattice points are of the form ${\bf e}_i + {\bf e}_j$ for all $1 \leq i \leq j \leq n$.  Furthermore, the set of points that have a particular $\sigma$ label forms a maximal connected open cell in ${\mathbb R}^n$.  We conclude that the ideal generated by the monomials corresponding to maximal cells of $X_{\mathcal A}$ is indeed the ideal $M_n^{(1)}$.

We next apply Lemma \ref{lem:cocell} to show that the complex of $S$-modules induced by the cochain complex of $X_{\mathcal A}$ along with this monomial labeling is indeed ${\mathbb K}$-acyclic.  For this suppose $m \in {\mathbb Z}^n$ is any monomial (identified with its exponent vector) and consider the complex $(X_{\mathcal A})_{\leq m}$.  We follow the argument in \cite{DocJosSan} to show that $(X_{\mathcal A})_{\leq m}$ is tropically convex, which implies that it is in fact contractible \cite{DevStu} (and in particular ${\mathbb K}$-acyclic).

The key observation here is that the monomial label of any point ${\bf x}$ in ${\mathcal A}$ can be read off from its type data.   Namely, if $F$ is a cell of ${\mathcal A}$ (of any dimension), and ${\bf x} \in F$ has type $T_{\mathcal A}({\bf x}) = (T_a, T_b)$ then the monomial label of $F$ is given by
\begin{equation}\label{eq:monlabel}
m_F = \prod_{i=1}^n x_i^{d_i},
\end{equation}
where 
\[d_i  =
\begin{cases}
      n  & \text{if }i \in T_a \text{ and } i \in T_b, \\
      n-1 & \text{if } i \in T_a \backslash T_b \text{ or } i \in T_b \backslash T_a, \\
      0 & \text{if } i \notin T_a \cup T_b.
      \end{cases}
     \]
This follows from the way that we have defined the monomial labeling.  Namely, if $F \in X_{\mathcal A}$ is a maximal cell and ${\bf x} \in F$ then ${\bf x}$ does not lie on either hyperplane of the arrangement and hence is in exactly one cone defined by each.  In this case the monomial label of the relevant cell is $x_i^{n-1}x_j^{n-1}$ if the cones are distinct, or $x_i^{n}$ if the cones are the same.  If ${\bf x}$ lies in some other (non maximal) cell $F \in X_{\mathcal A}$ then Proposition \ref{prop:types} implies that the monomial label $x_F$ described in (\ref{eq:monlabel}) is equal to the $\lcm$ of the monomial labels of all faces that contain $F$.

Now suppose ${\bf x}, {\bf y} \in (X_{\mathcal A})_{\leq m}$ and suppose ${\bf z}$ in a point in the tropical convex hull of ${\bf x}$ and ${\bf y}$. From Lemma \ref{lem:types} we have that $T_{\mathcal A}({\bf z})$ is contained in $T_{\mathcal A}({\bf x}) \cup T_{\mathcal A}({\bf y})$.  Hence if ${\bf z}$ lies in the cell $G \in {\mathcal X}_{\mathcal A}$ we have from (\ref{eq:monlabel}) that $m_F$ divides $m$.  We conclude that ${\bf z} \in (X_{\mathcal A})_{\leq m}$ so that $(X_{\mathcal A})_{\leq m}$ is tropically convex and hence ${\mathbb K}$-acyclic.

We next turn to minimality of the resolution.  If $F \subset G$ are cells of ${\mathcal X}_{\mathcal A}$ with $F \neq G$ then by Proposition \ref{prop:distinct} along with the labeling formula (\ref{eq:monlabel}) we have that $m_F \neq m_G$.  By Lemma \ref{lem:cocell} this implies that the resolution is minimal.
\end{proof}

Recall that the $i$th \defi{total Betti number} of a module is the sum $\sum_\sigma \beta_{i,\sigma}$ of all Betti numbers in the $i$th homological degree.  According to Theorem \ref{thm:resoneskel} the total Betti number of a $M_n^{(1)}$ is given by the number of codimension $i$ faces in the decomposition of ${\mathbb R}^n$ determined by the tropical hyperplane arrangement. Importing results from \cite{DocJosSan} we get the following formula.

\begin{cor}\label{cor:totalBetti}
The total Betti numbers $\beta_i^n$ of the ideal $M_{n}^{(1)}$ are given by
\[\beta_i^n = \sum_{j = 1}^n \; j {j-1 \choose i-1}\]
\end{cor}

Note that $\beta_1^n = \sum_{j=1}^n j = {n + 1 \choose 2}$, as predicted by Equation \ref{eq:gens}.

\begin{example}\label{ex:M3oneskel}
For $n=3$ the resolution of $M_3^{(1)}$ has a ${\mathbb Z}$-graded resolution given by 
\[0 \leftarrow M_{3}^{(2)} \leftarrow S(-3)^3 \oplus S(-4)^3 \leftarrow S(-5)^6 \oplus S(-6)^2 \leftarrow S(-7)^3 \leftarrow 0. \]
\noindent
The rank of each module, along with the structure of each differential map, can be read off from the monomial labeled complex depicted in Figure \ref{fig:labelcomplex}.  The total Betti numbers are given by $\beta_1^3 = 6$, $\beta_2^3 = 8$, $\beta_3^3 = 3$.
\end{example}

\begin{rem}
Note that since the complexes $(X_{\mathcal A})_{\leq m}$ considered above are \emph{contractible} (and hence ${\mathbb K}$-acylic for any field ${\mathbb K}$), it follows that the resolution and associated Betti numbers are independent of which field we are working over.  Similar properties hold for the resolutions of $M_G$ considered in \cite{PosSha} and \cite{DocSan}.
\end{rem}

The combinatorics of the decomposition of ${\mathbb R}^n$ induced by a generic arrangement of a pair of tropical hyperplanes is closely related to regular triangulations of a product of simplices $\Delta_1 \times \Delta_{n-1}$, as spelled out in \cite{DevStu}.   By the Cayley Trick \cite{DelRamSan} this implies that  a \emph{cellular} resolution of the ideal $M_{n}^{(1)}$ is supported on any `regular fine mixed subdivision' of the Minkowski sum $\Delta_{n-1} + \Delta_{n-1}$ (see Figure \ref{fig:mixedsub}).  We will not stress this perspective here and refer to \cite{DocJosSan} for more details, where arrangements of tropical hyperplanes are used to construct minimal cellular resolutions of similar monomial ideals arising from oriented matroid `type' data.

\begin{figure}[h] 
\includegraphics[scale = .45]{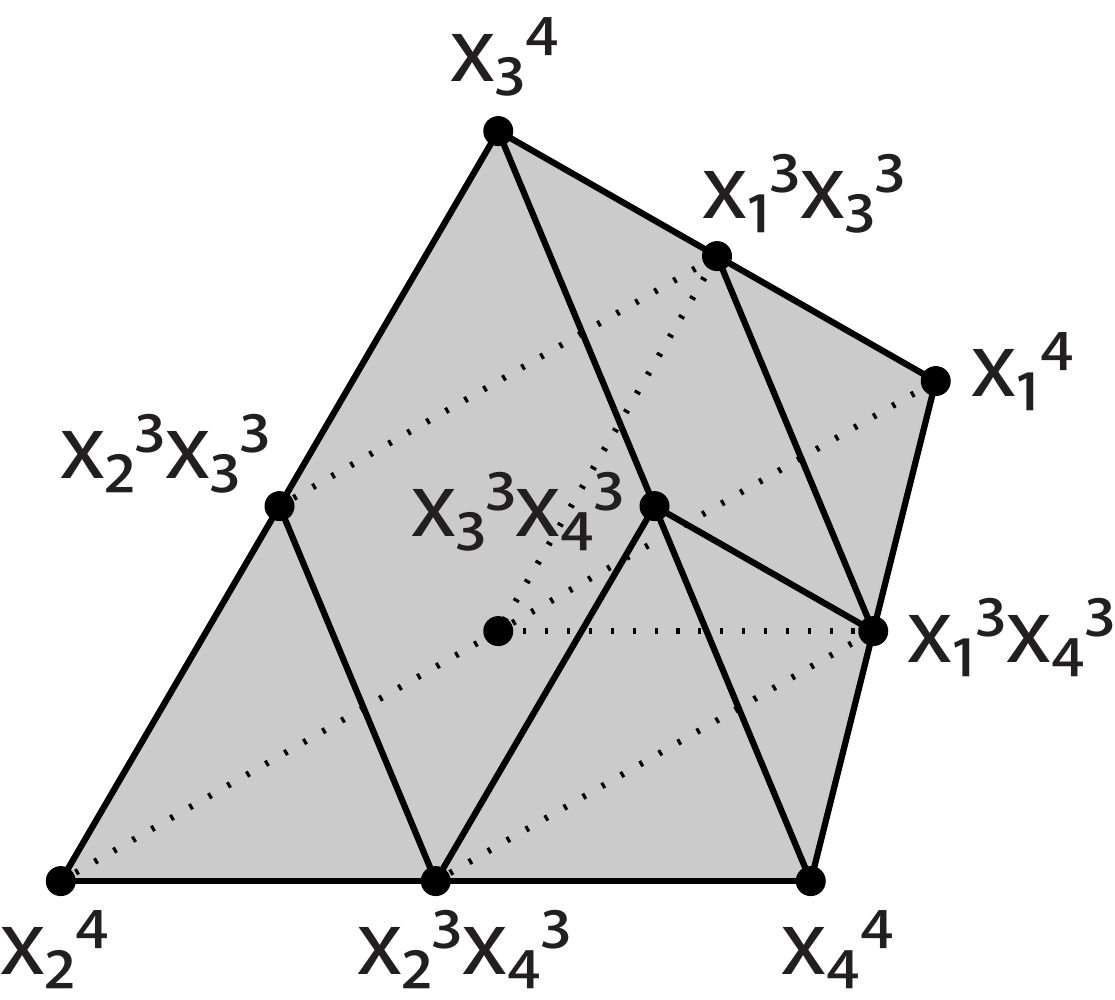}
\caption{A \emph{cellular} minimal resolution of $M_4^{(1)}$ is supported on any regular fine mixed subdivision of $\Delta_3 + \Delta_3$ (the vertex in the back is labeled $x_1^3x_2^3$). }
\label{fig:mixedsub}
\end{figure}

\begin{rem}
Triangulations of a product of simplices (and the related mixed subdivisions of dilated simplices) are widely studied objects (see for example Chapter 6.2 of \cite{DelRamSan}). Applying this technology we obtain several combinatorial interpretations for the Betti numbers of $M_n^{(1)}$.  In particular if we consider the `staircase' (or `pulling') triangulation of $\Delta_1 \times \Delta_{n-1}$ we see that maximal syzygies (corresponding to simplices in the triangulation) can be encoded by certain lattice paths or, equivalently, certain `non-crossing' spanning trees of the complete bipartite graph $K_{2,n}$ (see Figure \ref{fig:trees}).  We refer to \cite{DelRamSan} for more details.
\end{rem}

\begin{figure}[h] 
\includegraphics[scale = .45]{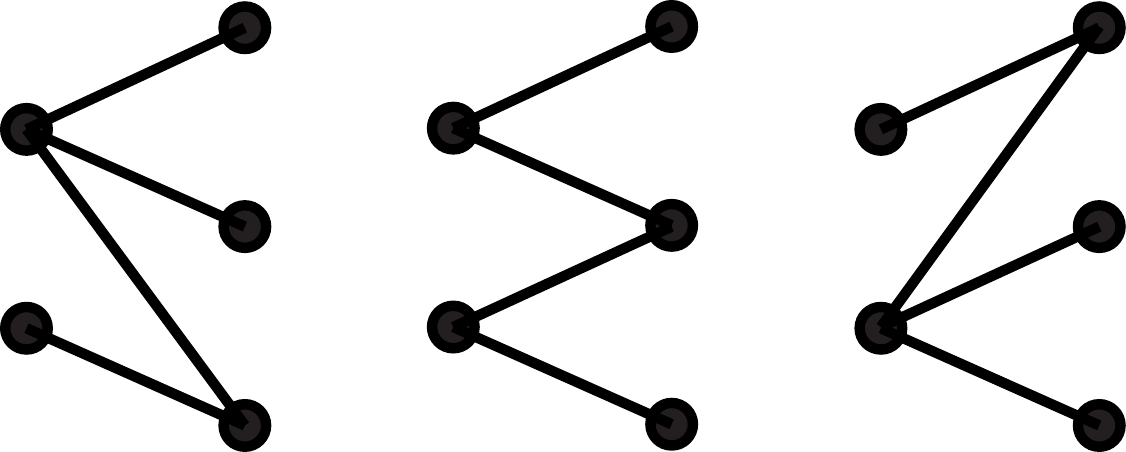}
\caption{Maximal syzygies of $M_n^{(1)}$ are indexed by `non-crossing' spanning trees of $K_{2,n}$ (here $n=3$, compare with Example \ref{ex:M3oneskel}). }
\label{fig:trees}
\end{figure}

We next consider resolutions of the ideal $M_G^{(1)}$ for an arbitrary graph $G$.  Recall that a (possibly redundant) set of generators of $M_G^{(1)}$ is again given by the set of monomials $m_\sigma$ for $1 \leq |\sigma| \leq 2$, and hence the decomposition of ${\mathbb R}^n$ induced by a generic arrangement of a pair of tropical hyperplanes has a natural monomial labeling with the property that its maximal cells correspond to these generators. However one can check that with such a labeling a generic arrangement of hyperplanes does not support a resolution of $M_G$ for general $G$.  For example if $G$ is the graph obtained from $K_4$ by removing the edge $12$, then the arrangement depicted in Figure \ref{fig:labelcomplex} (with relevant monomial labels) will have the property that the downset $({\mathcal X_{\mathcal A}})_{\leq (112)}$ is disconnected. For certain graphs, however, we can perturb our arrangements to obtain complexes that support (minimal) resolutions.

\begin{prop}\label{prop:resoneskeldef}
Suppose $G = H \ast \{0\}$ is obtained by coning the vertex $0$ over the graph $H$, where $H$ is obtained by removing a collection of disjoint cliques from the complete graph $K_n$.  Then there exists a (degenerate) arrangement ${\mathcal A} = (T({\bf a}), T({\bf b})) \subset {\mathbb R}^{n-1}$ of tropical hyperplanes with the property that the induced subdivision of ${\mathbb R}^{n-1}$ supports a minimal cocellular resolution $M_G^{(1)}$.
\end{prop}

\begin{proof}
Without loss of generality we can assume that $H$ is obtained from the complete graph $K_n$ by removing cliques $C_1, C_2, \dots, C_d$, where
\[V(C_1) = \{1,2, \dots, k_1\}, \; V(C_2) = \{k_1+1, \dots, k_2\}, \; \dots, \; V(C_d) = \{k_{d-1}+1, \dots, k_d\}.\]

We will explicitly write down the coordinates of the hyperplane arrangement that supports our resolution of $M_G^{(1)}$.  For this let ${\bf a} = (0, \dots, 0)$ be the origin in ${\mathbb R}^{n-1}$ and define ${\bf b} \in {\mathbb R}^{n-1}$ according to:
\[{\bf b} = 
\begin{cases}
(\underbrace{1, \dots, 1}_\text{$k_1$-times}, \underbrace{2, \dots, 2}_\text{$k_2-k_1$}, \dots, \underbrace{d, \dots, d}_\text{$k_d - k_{d-1}$}, d+1, d+2, \dots, n-k_d+d), &\quad\text{if $k_d \neq n$} \\

(\underbrace{1, \dots, 1}_\text{$k_1$-times}, \underbrace{2, \dots, 2}_\text{$k_2-k_1$}, \dots, \underbrace{d-1, \dots, d-1}_\text{$k_{d-1} - k_{d-2}$}, \underbrace{0,0, \dots, 0}_\text{$k_d - k_{d-1} - 1$}),  &\quad\text{if $k_d = n$.} \\
\end{cases}
\]

Let ${\mathcal A} = (T({\bf a}), T({\bf b}))$ be the resulting arrangement of tropical tropical hyperplanes, and let ${\mathcal X}_{\mathcal A}$ denote the induced decomposition of ${\mathbb R}^{n-1}$.  We first claim that the maximal regions of ${\mathcal X}_{\mathcal A}$ are in one-to-one correspondence with the set of minimal generators of $M_G^{(1)}$.  Recall that for any $G$, a redundant set of generators for $M_G^{(1)}$ is given by all monomials $m_\sigma$ for all subsets $\sigma \subset [n]$ with $1 \leq |\sigma| \leq 2$.  For our given graph $G$, a \emph{minimal} set of generators is given by monomials $m_{\{i\}} = x_i^{\deg i}$ for all $i \in [n]$, along with $m_{\{i,j\}}$ for all subsets $\{i,j\} \in E(G)$.  To see this note that if $ij \notin E(G)$ we have
\[m_{\{i,j\}} = x_i^{\deg(i)}x_j^{\deg(j)} = m_{\{i\}}m_{\{j\}},\]
\noindent
and hence the monomial $m_{\{i,j\}}$ is redundant.  

On the other hand, if $F \in {\mathcal X}_{\mathcal A}$ is a maximal cell with ${\bf x} \in F$ (so that ${\bf x}$ does not intersect either of the hyperplanes $T({\bf a})$ nor $T({\bf b})$) then ${\bf x}$ sits in exactly one cone of each hyperplane.  Hence the sets $T_a$ and $T_b$ are each singletons (where $T({\bf x}) = (T_a, T_b)$ is the type of ${\bf x}$).  By construction the only such types that can occur are of the type $(\{i\}, \{i\})$ and $(\{i\},\{j\})$, where $ij \in E(G)$.  Furthermore, all such pairs of singletons actually do occur exactly once as types of some region.  

To see this, note that since these types occur exactly once in a \emph{generic} arrangement it suffices to show that these types are actually achieved in our particular arrangement ${\mathcal A} = (T({\bf a}), T({\bf b}))$.  We first consider the repeated singletons. To achieve $(\{i\}, \{i\})$ for $i \neq n$ simply take the ${\bf x}$ to be a vector of 0's except for a sufficiently large positive entry in the $i$th coordinate.  To achieve $(\{n\},\{n\})$ take ${\bf x} = (-1,-1, \dots, -1)$, a vector of all $-1$'s.

To realize a cell with the type $(\{i\},\{j\})$ with $i<j$ we have to consider a couple of cases depending on the value of $k_d$ (the largest vertex in our removed cliques). In the case that $k_d \neq n$ we achieve the type $(\{i\},\{j\})$,  $j \neq n$ by taking ${\bf x}$ to be the vector consisting of all 0's except the values $b_i + 1$ in the $i$th coordinate and $b_j + 1/2$ in the $j$th coordinate.  To achieve the pair $(\{i\},\{n\})$ take ${\bf x}$ to be all 0's except for the value $b_i + 1/2$ in the $i$th coordinate.

In the case that $k_d = n$ we can achieve any type $(\{i\}, \{j\})$ with $i < j$ as follows.  If $j \in C_d$ take ${\bf x}$ to be all 0's except $b_i + 1/2$ in the $i$th coordinate and $b_j + 1 = 1$ in the $j$th coordinate.  Otherwise if $j \notin C_d$ take ${\bf x}$ to be all 0's except $b_i + 1$ in the $i$th coordinate and $b_j + 1/2$ in the $j$th coordinate.

We next describe the mononial labeling on the complex ${\mathcal X}_{\mathcal A}$, a modification of the monomial labeling of the case of the complete graph described in the proof of Theorem \ref{thm:resoneskel}.  If $F$ is a cell of ${\mathcal A}$ (of any dimension), and ${\bf x} \in F$ has type $T_{\mathcal A}({\bf x}) = (T_a, T_b)$ then the monomial label of $F$ is given by
\begin{equation}\label{eq:monlabelnongen}
m_F = \prod_{i=1}^n x_i^{d_i},
\end{equation}
where 
\[d_i  =
\begin{cases}
      \deg(i)   & \text{if }i \in T_a \text{ and } i \in T_b, \\
      \deg(i) - 1 & \text{if } i \in T_a \backslash T_b \text{ or } i \in T_b \backslash T_a, \\
      0 & \text{if } i \notin T_a \cup T_b.
      \end{cases}
     \]
With this convention, the maximal cells of ${\mathcal X}_{\mathcal A}$ are labeled with monomials corresponding to the set of minimal generators of $M_G^{(1)}$.   For all other faces $F$, Proposition \ref{prop:types} implies that the monomial label $x_F$ described in (\ref{eq:monlabel}) is equal to the $\lcm$ of the monomial labels of all faces that contain $F$.  If $m \in {\mathbb Z}^n$ is any monomial then the same argument employed in the proof of Theorem \ref{thm:resoneskel} shows that the complex $({\mathcal X}_{\mathcal A})_{\leq m}$ is tropically convex and hence contractible.  Applying Lemma \ref{lem:cocell} then implies that ${\mathcal X}_{\mathcal A}$ supports a cocellular resolution of the ideal $M_G^{(1)}$.  Furthermore, if $F \subset G$ are cells of ${\mathcal X}_{\mathcal A}$ with $F \neq G$ then Proposition \ref{prop:distinct} along with the formula (\ref{eq:monlabelnongen}) implies that $m_F \neq m_G$, so that the resolution is minimal.   This completes the proof.
\end{proof}

\begin{figure}[h]
\includegraphics[scale = .4]{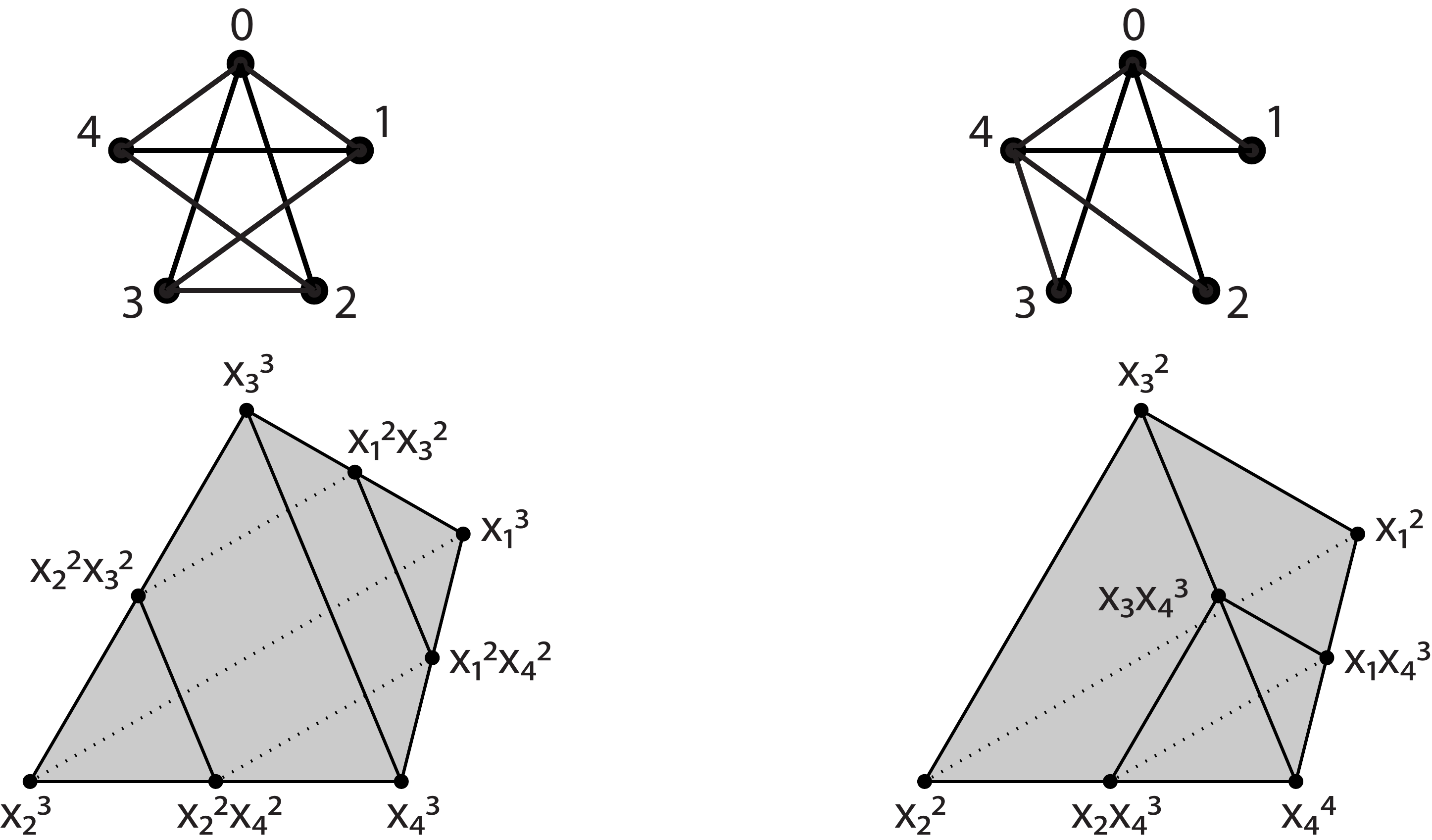}
\caption{Resolutions of $M_{G_i}^{(1)}$ for the graphs $G_1$ and $G_2$ from Example \ref{ex:degens}.}
\label{fig:degenres}
\end{figure}

\begin{example}\label{ex:degens}
As examples of the previous result, consider the graphs $G_1$ and $G_2$  depicted in Figure \ref{fig:degenres}. The graph $G_1$ is obtained by removing the complete graphs (edges) $12$ and $34$, and in $G_2$ we have removed the triangle $123$.  The cellular resolutions of each ideal are illustrated in Figure (\ref{fig:degenres}).  In both cases the complexes are dual to a non-generic arrangement of a pair of tropical hyperplanes in ${\mathbb R}^3$. For $G_1$ we have ${\bf b} = (1,1,0)$, for $G_2$ we have ${\bf b} = (1,1,1)$.
\end{example}

\section{Discussion and open questions}
We end with some discussion regarding other open questions and possible future directions.

\subsection{Resolutions of other skeleta}
We do not have uniform descriptions of resolutions for other skeleton ideals $M_G^{(k)}$, even for the complete graph $G = K_{n+1}$.  We have seen that a cellular resolution of $M_n = M_{K_{n+1}}$ is supported on ${\mathcal B}(\Delta_{n-1})$, the barycentric subdivision of the $(n-1)$-dimensional simplex.  As discussed above, a reasonable candidate for a geometric complex supporting a resolution of the ideals $M_n^{(k)}$ would be a subdivision of the simplex $\Delta_{n-1}$ that agrees with ${\mathcal B}(\Delta_{n-1})$ up to the $k$-dimensional skeleton.   For the case of $k=1$ we have seen that complex generated by a generic arrangement of two tropical hyperplanes has this property, but we do not know of any candidates for other skeleta.  We remark that an arrangement of $k$ tropical hyperplanes will \emph{not} produce the desired complex, since for instance the number of maximal cells is not equal to the number of minimal generators of the relevant ideal.

A particularly tractable case seems to be in codimension one, where the ideal $M^{(n-2)}_G \subset M_G$ is only missing the generator of $M_G$ given by the monomial $m_{[n]} $.   Here we would need a subdivision of $\Delta_{n+1}$ that agrees with the complex ${\mathcal B}(\Delta_{n+1})$ on the boundary but is missing the `interior' vertex.

\subsection{Enumeration of parking functions}

A natural invariant of a parking function (thought of as a standard monomial in $M_n \coloneqq M_{K_{n+1}}$) is its degree.  With this statistic we can define a generating function $P_n(q)$ of parking functions of length $n$ as follows:
\[P_n(q) = \sum_{\alpha = (a_1, a_2, \dots, a_n)} q^{a_1 + a_2 + \dots +a_n}\]
where $\alpha$ ranges over all parking functions of length $n$ (corresponding to standard monomials of $M_n$). For small values of $n$ we get
\begin{equation}
\begin{split}
& P_1(q) = 1 \\
& P_2(q) = 2q + 1\\
& P_3(q) = 6q^3 + 6q^2 + 3q + 1\\
& P_4(q) = 24q^6 + 36q^5 + 30q^4 + 20q^3 + 10q^2 + 4q + 1
\end{split}
\end{equation}

Kreweras \cite{Kre} studied another polynomial $I_n(q)$ which enumerates (labeled) rooted forests by number of inversions, and showed that $I_n(q)$ was closely connected to $P_n(q)$.  
Here a \defi{rooted forest} on $[n]$ is a graph on the vertex set $\{1,2,\dots, n\}$ with the property that each connected component is a rooted tree.  An \defi{inversion} of a rooted forest is a pair $(i,j)$ satisfying $i < j$ and such that $j$ lies on the unique path connecting $k$ to $i$, where $k$ is the root of the tree containing the vertex $i$.  We let $\inv(F)$ denote the number of inversions of $F$ and define the polynomial $I_n(q)$ as
\[I_n(q) = \sum_F q^{\inv(F)}\]
where $F$ ranges over all labeled forests on $[n]$.

The polynomials $I_n(q)$ have the following connection to parking functions of complete graphs (as shown in \cite{Kre}):
\[q^{n \choose 2} I_n(1/q) = P_n(q).\]

Is there an analogous story for the standard monomials of the skeleton ideals $M_n^{(k)}$?  Using the degree statistic we can similarly define a generating function 
\[P^{(k)}_n(q) = \sum_{\beta = (b_1, b_2, \dots, b_n)} q^{b_1 + b_2 + \dots +b_n}\]
where $\beta$ ranges over all standard monomials of $M_n^{(k)}$.

For the case $k=1$ and small values of $n$ we have
\begin{equation}
\begin{split}
& P^{(1)}_1(q) = 1 \\
& P^{(1)}_2(q) = 2q + 1\\
& P^{(1)}_3(q) = 3q^4 + 7q^3 + 6q^2 + 3q + 1\\
& P^{(1)}_4(q) = 4q^9 + 13q^8 + 28q^7 + 38q^6 + 40q^5 + 31 q^4 + 20 q^3 + 10 q^2 + 4q +1
\end{split}
\end{equation}

In Section \ref{sec:stdmon} we have seen that standard monomials of $M_n^{(1)}$ are counted (in a weighted fashion) by the spanning $TU$-subgraphs of $K_n$ (with spanning trees appearing as a special cases).  Is there a notion of an `inversion' for a $TU$-subgraph that recovers this statistic?

For the case $k = n-2$ the standard monomials of the ideals $M_n^{(n-2)}$ lead to a notion of `spherical parking functions' that are shown in \cite{DocPar} to be in bijection with a class of labeled graphs called `uprooted trees'.  In \cite{DocPar} a conjectural relationship between the degree of spherical parking functions and a notion of inversions for uprooted trees is also provided.

\subsection{Power ideals}

Recall that in \cite{PosSha} the authors were interested in another class of `power ideals' that one associate to a graph $G$ on vertex set $\{0,1,\dots,n\}$.  We briefly recall the definition.  For any nonempty subset $\sigma \subseteq \{1,2,\dots, n\}$ let
\[p_\sigma = \big(\sum_{i \in \sigma} x_i\big)^{D_\sigma},\]
\noindent
where $D_\sigma = \sum_{i \in \sigma} d_\sigma(i)$, with $d_\sigma$ defined in Equation (\ref{eq:dsigma}).  Now define $J_G$ to be the ideal in the polynomial ring $S = {\mathbb K}[x_1, x_2, \dots, x_n]$ generated by the $p_\sigma$ for all nonempty subsets $\sigma$, and let ${\mathcal B}_G$ denote the quotient algebra ${\mathcal B}_G = S/J_G$.  Recall that ${\mathcal A}_G = S/M_G$ is the algebra obtained by modding out the $G$-parking function ideal.  The algebras ${\mathcal A}_G$ and ${\mathcal B}_G$ are both graded, and in \cite{PosSha} it is shown that the Hilbert series coincide.  In fact both algebras have a ${\mathbb K}$-basis given by the monomials ${\bf x}^{\bf b}$ where ${\bf b}$ ranges over all $G$-parking functions.   This is particularly surprising since (in general) the monomial ideal $M_G$ is not an \emph{initial ideal} of the ideal $J_G$ for any monomial order, and hence standard Gr\"obner degeneration techniques do not apply.

In our context we have a natural definition for the skeleton of these power ideals, mimicking the construction of the monomial case discussed above.  Namely, for any graph $G$ and integer $k$ with $0 \leq k \leq n-1$ we define the $k$-skeleton power ideal $J^{(k)}_G$ to be the subideal of $J_G$ given by
\[J_G^{(k)} = \langle p_\sigma: 1 \leq |\sigma| \leq k+1 \rangle.\]

One can check that for $0 \leq k \leq 1$ we have $M^{(k)}_G = J^{(k)}_G$ (equality as ideals), but already for the 2-skeleton the ideals differ.  In fact the quotient algebras can have different Hilbert series, as the next example illustrates.

\begin{example} \label{ex:Hilbineq}
Let $G$ be the graph obtained from removing the edge $12$ from the complete graph on vertex set $\{0,1,2,3,4\}$.  Then we have
\begin{equation}
M_G^{(2)} = \langle x_1^3, x_2^3, x_3^4, x_4^4, x_1^2x_3^3, x_1^2x_4^3, x_2^2x_3^3, x_2^2x_4^3, x_3^3x_4^3, x_1^2x_2^2x_3^2, x_1^2x_2x_4^2, x_1x_3^2x_4^2, x_2x_3^2x_4^2 \rangle \\
\end{equation}

\begin{equation}
\begin{aligned}
\quad J_G^{(2)} = {} &\langle x_1^3, x_2^3, x_3^4, x_4^4, (x_1 + x_2)^6, (x_1+ x_3)^5, (x_1+x_4)^5, (x_2+x_3)^5, (x_2+x_4)^5,  \\
 & (x_3+x_4)^6, (x_1 + x_2 + x_3)^6, (x_1 +x_2 + x_4)^5, (x_1+ x_3+x_4)^5, (x_2 + x_3 + x_4)^5 \rangle
 \end{aligned}
 \end{equation}

According to our calculations with Macaulay2 \cite{M2}, the graded components of $S/M_G^{(2)}$ and $S/J_G^{(2)}$ have the same dimension except in degree 6, where
\[\dim_{\mathbb K} (S/M_G^{(2)})_6 = 7 > 6 = \dim_{\mathbb K} (S/J_G^{(2)})_6.\]

\end{example} 
We do not know if termwise inequalities $\Hilb S/M_G^{(k)} \geq \Hilb S/J_G^{(k)}$ hold for the Hilbert functions of the two algebras, as was the case for deformations of more general monotone monomial ideals \cite{PosSha} (and as in Example \ref{ex:Hilbineq}).  We will point that in this example the monomial basis for $(S/M_G^{(2)})_6$ does \emph{not} span the linear space $(S/J_G^{(2)})_6$.


\begin{thebibliography}{30}


\bibitem{Bac}
S.~Backman, \emph{A Bijection Between the Recurrent Configurations of a Hereditary Chip-Firing Model and Spanning Trees}, preprint available at \texttt{arxiv.org:1401.3309}.


\bibitem{Bap}
R.~P.~Bapat, \emph{Graphs and matrices}, Universitext, Springer-Verlag London, 2010.


\bibitem{BayStu}
D. Bayer, B.Sturmfels, \emph{Cellular resolutions of monomial modules}, J. Reine Angew. Math. {\bf 502} (1998), pp. 123--140.


\bibitem{CarPaoSpo}
S. Caracciolo, G. Paoletti, A. Sportiello, \emph{Multiple and inverse topplings in the Abelian Sandpile Model}, Eur. Phys. J. Special Topics {\bf 212} (2012).
 

\bibitem{ChePyl}
D. Chebikin, P.Pylyavskyy, \emph{A family of bijections between G-parking functions and spanning trees}, J. Combin. Theory Ser. A {\bf 110}, Issue 1 (2005), pp. 31--41.



\bibitem{CorLeB}
R. Cori, Y. Le Borgne, \emph{The sand-pile model and Tutte polynomials}, Adv. Appl. Math. {\bf 30} (2003), pp. 44--52.

\bibitem{DelRamSan}
J.~De Loera, J.~Rambau, F.~Santos, \emph{Triangulations, Structures for Algorithms and Applications}, Algorithms and Computation in Mathematics, vol. 25, Springer-Verlag Berlin Heidelberg, 2010.


\bibitem{DevStu}
M.~Develin, B.~Sturmfels, \emph{Tropical convexity}, Doc. Math. {\bf 9} (2004), pp. 1--27. 


\bibitem{DocPar}
A. Dochtermann, \emph{Spherical parking functions, uprooted trees, and yet another way to count $n^n$}, preprint.

\bibitem{DocSan}
A. Dochtermann and R. Sanyal, \emph{Laplacian ideals, arrangements, and resolutions}, J. Algebr. Comb. (2014) {\bf 40},  pp. 805--822.

\bibitem{DocJosSan}
A.~Dochtermann, M.~Joswig, R.~Sanyal, \emph{Tropical types and associated cellular resolutions}, J. Algebra. {\bf 356} (2012), pp. 304--324.




\bibitem{Gab}
 A. Gabrielov, \emph{Abelian avalanches and Tutte polynomials}, Phys. A {\bf 195} (1993), no. 1-2, pp. 253--274.
 
 
 \bibitem{M2}
 D.~Grayson, M.~Stillman, \emph{Macaulay2, a software system for research in algebraic geometry}, Available at \texttt{http://www.math.uiuc.edu/Macaulay2/}
        
        

\bibitem{Kre}
G. Kreweras, \emph{Une famille de polyn\^omes ayant plusieurs propri\'et\'es \'enumeratives}. (French) Period. Math. Hungar. {\bf 11} (1980), no. 4, pp. 309--320.
 

\bibitem{MilStu}
E.~Miller, B.~Sturmfels, \emph{Combinatorial commutative algebra}. Graduate Texts in Mathematics {\bf 227}. Springer-Verlag, New York, 2005. 

\bibitem{PitSta}
J. Pitman, R. P. Stanley, \emph{A polytope related to empirical distributions, plane trees, parking functions, and the associahedron}, Discrete Comput. Geom., {\bf 27} (2002), pp. 603--634.

\bibitem{PosSha}
A. Postnikov, B. Shapiro, \emph{Trees, parking functions, syzygies, and deformations of monomial ideals}, Trans. Amer. Math. Soc. {\bf 356} (2004), pp. 3109--3142.

\bibitem{Yan}
C. H.  Yan, \emph{On the enumeration of generalized parking functions}, Proceedings of the Thirty-first Southeastern International Conference on Combinatorics, Graph Theory and Computing (Boca Raton, FL, 2000), 2000, pp. 201--209.

\bibitem{YanGen}
C. H. Yan, \emph{Generalized Parking Functions, Tree Inversions and Multicolored Graphs},  Adv. Appl. Math. {\bf 27} (2001), pp. 641--670.



\end{thebibliography}
\end{document}